\documentclass[10pt]{amsart}
\usepackage{amssymb,MnSymbol}
\usepackage{amsthm,amsmath}
\usepackage{stmaryrd}

\title[Axiomatizations of Lov\'asz extensions]{Axiomatizations of Lov\'asz extensions of pseudo-Boolean functions}

\author{Miguel Couceiro}
\address{Mathematics Research Unit, FSTC, University of Luxembourg \\
6, rue Coudenhove-Kalergi, L-1359 Luxembourg, Luxembourg}%
\email{miguel.couceiro[at]uni.lu }

\author{Jean-Luc Marichal}
\address{Mathematics Research Unit, FSTC, University of Luxembourg \\
6, rue Coudenhove-Kalergi, L-1359 Luxembourg, Luxembourg}%
\email{jean-luc.marichal[at]uni.lu }

\date{May 5, 2011}

\begin{document}

\theoremstyle{plain}
\newtheorem{theorem}{Theorem}%[section]% Supprimer [section] pour une numérotation linéaire
\newtheorem{lemma}[theorem]{Lemma}
\newtheorem{proposition}[theorem]{Proposition}
\newtheorem{corollary}[theorem]{Corollary}
\newtheorem{fact}[theorem]{Fact}
\newtheorem*{main}{Main Theorem}

\theoremstyle{definition}
\newtheorem{definition}[theorem]{Definition}
\newtheorem{example}[theorem]{Example}

\theoremstyle{remark}
\newtheorem*{conjecture}{Conjecture}
\newtheorem{remark}{Remark}
\newtheorem{claim}{Claim}

\newcommand{\N}{\mathbb{N}}
\newcommand{\R}{\mathbb{R}}
\newcommand{\Vspace}{\vspace{2ex}}
\newcommand{\bfx}{\mathbf{x}}

\begin{abstract}
Three important properties in aggregation theory are investigated, namely horizontal min-additivity, horizontal max-additivity, and comonotonic
additivity, which are defined by certain relaxations of the Cauchy functional equation in several variables. We show that these properties are
equivalent and we completely describe the functions characterized by them. By adding some regularity conditions, these functions coincide with
the Lov\'asz extensions vanishing at the origin, which subsume the discrete Choquet integrals. We also propose a simultaneous generalization of
horizontal min-additivity and horizontal max-additivity, called horizontal median-additivity, and we describe the corresponding function class.
Additional conditions then reduce this class to that of symmetric Lov\'asz extensions, which includes the discrete symmetric Choquet integrals.
\end{abstract}

\keywords{Aggregation function, discrete Choquet integral, discrete symmetric Choquet integral, Lov\'asz extension, functional equation, Cauchy
equation, comonotonic additivity, horizontal additivity, axiomatization}

\subjclass[2010]{Primary 39B22, 39B72; Secondary 26B35}

\maketitle

%---------------------------------------------------------------------------------------------- Section 1
\section{Introduction}

When we need to merge a set of numerical values into a single one, we usually make use of a so-called aggregation function, e.g., a mean or an
averaging function.
%The need to aggregate values in a meaningful way has become more and more present in an increasing number of areas not only
%of mathematics or physics, but especially in applied fields such as engineering, computer science, and economical and social sciences.
Various aggregation functions have been proposed in the literature, thus giving rise to the growing theory of aggregation which proposes,
analyzes, and characterizes aggregation function classes. For recent references, see Beliakov et al.~\cite{BelPraCal07} and Grabisch et
al.~\cite{GraMarMesPap09}.

A noteworthy aggregation function is the so-called discrete Choquet integral, which has been widely investigated in aggregation theory, due to
its many applications for instance in decision making (see the edited book \cite{GraMurSug00}). A convenient way to introduce the discrete
Choquet integral is via the concept of Lov\'asz extension. An $n$-place Lov\'asz extension is a continuous function $f\colon\R^n\to\R$ whose
restriction to each of the $n!$ subdomains
$$
\R^n_{\sigma}=\{\bfx=(x_1,\ldots,x_n)\in\R^n : x_{\sigma(1)}\leqslant\cdots\leqslant x_{\sigma(n)}\}\qquad (\sigma\in S_n)
$$
is an affine function, where $S_n$ denotes the set of permutations on $[n]=\{1,\ldots,n\}$. An $n$-place Choquet integral is simply a
nondecreasing (in each variable) $n$-place Lov\'asz extension which vanishes at the origin. For general background, see
\cite[{\S}5.4]{GraMarMesPap09}.

The class of $n$-place Choquet integrals has been axiomatized independently by means of two noteworthy aggregation properties, namely
comonotonic additivity (see, e.g., \cite{deCBol92}) and horizontal min-additivity (originally called ``horizontal additivity'', see
\cite[{\S}2.5]{BenMesViv02}). Recall that a function $f\colon\R^{n}\to\R$ is said to be comonotonically additive if, for every $\sigma\in S_n$,
we have
$$
f(\bfx+\bfx') = f(\bfx)+ f(\bfx')\qquad (\bfx,\bfx'\in\R^n_{\sigma}).
$$
To describe the second property, consider the horizontal min-additive decomposition of the $n$-tuple $\bfx\in\R^n$ obtained by ``cutting'' it
with a real number $c$, namely
$$
\bfx = (\bfx\wedge c)+ (\bfx-(\bfx\wedge c)),
$$
where $\bfx\wedge c$ denotes the $n$-tuple whose $i$th component is $x_i\wedge c=\min(x_i,c)$. A function $f\colon\R^{n}\to\R$ is said to be
horizontally min-additive if
$$
f(\bfx) = f(\bfx\wedge c)+ f(\bfx-(\bfx\wedge c))\qquad (\bfx\in\R^n,\, c\in\R).
$$

In this paper we completely describe the function classes axiomatized by each of these properties. More precisely, after recalling the
definitions of Lov\'asz extensions, discrete Choquet integrals, and their symmetric versions (Section 2), we show that comonotonic additivity
and horizontal min-additivity (as well as its dual counterpart, namely horizontal max-additivity) are actually equivalent properties. We
describe the function class axiomatized by these properties and we show that, up to certain regularity conditions (based on those we usually add
to the Cauchy functional equation to get linear solutions only), these properties completely characterize those $n$-place Lov\'asz extensions
which vanish at the origin. Nondecreasing monotonicity is then added to characterize the class of $n$-place Choquet integrals (Section 3). We
also introduce a weaker variant of the properties above, called horizontal median-additivity, and determine the function class axiomatized by
this new property. Finally, by adding some natural properties, we characterize the class of $n$-place symmetric Lov\'asz extensions and the
subclass of $n$-place symmetric Choquet integrals (Section 4).

We employ the following notation throughout the paper. Let $\R_+=\left[0,+\infty\right[$ and $\R_-=\left]-\infty,0\right]$. We let $I$ denote a
nontrivial (i.e., of positive measure) real interval, possibly unbounded, containing the origin $0$. We also introduce the notation
$I_+=I\cap\R_+$, $I_-=I\cap\R_-$, and $I_{\sigma}^n=I^n\cap\R_{\sigma}^n$. For every $A\subseteq [n]$, the symbol $\mathbf{1}_A$ denotes the
$n$-tuple whose $i$th component is $1$, if $i\in A$, and $0$, otherwise. Let also $\mathbf{1}=\mathbf{1}_{[n]}$ and
$\mathbf{0}=\mathbf{1}_{\varnothing}$. The symbols $\wedge$ and $\vee$ denote the minimum and maximum functions, respectively. For every
$\bfx\in\R^n$, let $\bfx^+=\bfx\vee 0$ and $\bfx^-=(-\bfx)^+$. For every function $f\colon I^n\to\R$, we define its diagonal section
$\delta_f\colon I\to\R$ by $\delta_f(x)=f(x\mathbf{1})$. More generally, for every $A\subseteq [n]$, we define the function $\delta_f^A\colon
I\to\R$ by $\delta_f^A(x)=f(x\mathbf{1}_A)$.

In order not to restrict our framework to functions defined on $\R$, we consider functions defined on intervals $I$ containing $0$, in
particular of the forms $I_+$, $I_-$, and those centered at $0$.

It is important to notice that comonotonic additivity as well as horizontal min-additivity and horizontal max-additivity, when restricted to
functions $f\colon I^n\to\R$, extend the classical additivity property defined by the Cauchy functional equation for $n$-place functions
\begin{equation}\label{eq:sddfs78f}
f(\bfx+\bfx') = f(\bfx)+ f(\bfx')\qquad (\bfx,\bfx',\bfx+\bfx'\in I^n).
\end{equation}
In this regard, recall that the general solution $f\colon I^n\to\R$ of the Cauchy equation (\ref{eq:sddfs78f}) is given by
$f(\bfx)=\sum_{k=1}^nf_k(x_k)$, where the $f_k\colon I\to\R$ ($k\in [n]$) are arbitrary solutions of the basic Cauchy equation
$f_k(x+x')=f_k(x)+f_k(x')$ (see \cite[{\S}2--4]{AczDho89}). As the following theorem states, under some regularity conditions, each $f_k$ is
necessarily a linear function.

\begin{theorem}\label{thm:sa876as}
Let $I$ be a nontrivial real interval, possibly unbounded, containing $0$. If $f\colon I\to\R$ solves the basic Cauchy equation
$$
f(x+x')=f(x)+f(x')\qquad (x, x', x+x'\in I),
$$
then either $f$ is of the form $f(x) = cx$ for some $c\in\R$, or the graph of $f$ is everywhere dense in $I\times\R$. The latter case is
excluded as soon as $f$ is continuous at a point or monotonic or Lebesgue measurable or bounded from one side on a set of positive measure.
\end{theorem}

\begin{proof}
See Appendix~\ref{app}.
\end{proof}

\noindent (We would like to acknowledge Professor Maksa at the Institute of Mathematics of the University of Debrecen, Hungary, for providing the proof of Theorem~\ref{thm:sa876as}.)

As we will see in this paper, comonotonic additivity, horizontal min-additivity, and horizontal median-additivity of a function $f\colon
I^n\to\R$ force the $1$-place functions $\delta_f^A|_{I_+}$, and $\delta_f^A|_{I_-}$ $(A\subseteq [n])$ to solve the basic Cauchy equation.
Theorem~\ref{thm:sa876as} will hence be useful to describe the corresponding function classes whenever the regularity conditions stated are
assumed.

Recall that a function $f\colon I^n\to\R$ is said to be \emph{homogeneous} (resp.\ \emph{positively homogeneous}) \emph{of degree one} if
$f(c\,\bfx)=c\,f(\bfx)$ for every $\bfx\in I^n$ and every $c\in\R$ (resp.\ every $c>0$) such that $c\,\bfx\in I^n$.

%---------------------------------------------------------------------------------------------- Section 2
\section{Lov\'asz extensions and symmetric Lov\'asz extensions}

We now recall the concept of Lov\'asz extension and introduce that of symmetric Lov\'asz extension.

Consider a \emph{pseudo-Boolean function}, that is, a function $\phi\colon\{0,1\}^n\to\R$. The \emph{Lov\'asz extension} of $\phi$ is the
function $f_{\phi}\colon\R^n\to\R$ whose restriction to each subdomain $\R^n_{\sigma}$ $(\sigma\in S_n)$ is the unique affine function which
agrees with $\phi$ at the $n+1$ vertices of the $n$-simplex $[0,1]^n\cap\R^n_{\sigma}$ (see \cite{Lov83,Sin84}). We then have
$f_{\phi}|_{\{0,1\}^n}=\phi$.

We say that a function $f\colon\R^n\to\R$ is a \emph{Lov\'asz extension} if there is a function $\phi\colon\{0,1\}^n\to\R$ such that
$f=f_{\phi}$. For any Lov\'asz extension $f\colon\R^n\to\R$, the function $f_0=f-f(\mathbf{0})$ has the representation
\begin{equation}\label{eq:F}
f_0(\bfx) = x_{\sigma(1)}\,\delta_{f_0}(1)+\sum_{i=2}^{n} (x_{\sigma(i)}-x_{\sigma(i-1)})\,\delta_{f_0}^{A_{\sigma}^{\uparrow}(i)}(1)\qquad
(\bfx\in\R^n_{\sigma}),
\end{equation}
where $A_{\sigma}^{\uparrow}(i)=\{\sigma(i),\ldots,\sigma(n)\}$, with $A_{\sigma}^{\uparrow}(n+1)=\varnothing$. Indeed, both sides of (\ref{eq:F}) agree at $\bfx=\mathbf{0}$ and
$\bfx=\mathbf{1}_{A_{\sigma}^{\uparrow}(k)}$ for every $k\in [n]$. Thus we see that $f_0$ is positively homogeneous of degree one.

An $n$-place \emph{Choquet integral} is a nondecreasing Lov\'asz extension $f\colon\R^n\to\R$ such that $f(\mathbf{0})=0$. It is easy to see
that a Lov\'asz extension $f\colon\R^n\to\R$ is a Choquet integral if and only if its underlying pseudo-Boolean function $\phi=f|_{\{0,1\}^n}$
is nondecreasing and vanishes at the origin (see \cite[{\S}5.4]{GraMarMesPap09}).

We now introduce the concept of symmetric Lov\'asz extension. Here ``symmetric'' does not refer to invariance under a permutation of variables
but rather to the role of the origin of $\R^n$ as a symmetry center with respect to the function values whenever the function vanishes at the
origin.

The \emph{symmetric Lov\'asz extension} of a pseudo-Boolean function $\phi\colon\{0,1\}^n\to\R$ is the function
$\check{f}_{\phi}\colon\R^n\to\R$ defined by
\begin{equation}\label{eq:sadfdfsa76}
\check{f}_{\phi}(\bfx)=f_{\phi}(\mathbf{0})+f_{\phi}(\bfx^+)-f_{\phi}(\bfx^-)\qquad (\bfx\in\R^n),
\end{equation}
where $f_{\phi}$ is the Lov\'asz extension of $\phi$.

We say that a function $f\colon\R^n\to\R$ is a \emph{symmetric Lov\'asz extension} if there is a function $\phi\colon\{0,1\}^n\to\R$ such that
$f=\check{f}_{\phi}$. For any symmetric Lov\'asz extension $f\colon\R^n\to\R$, by (\ref{eq:F}) and (\ref{eq:sadfdfsa76}) the function
$f_0=f-f(\mathbf{0})$ has the representation
\begin{eqnarray}
f_0(\bfx) &=& x_{\sigma(p+1)}\,\delta_{f_0}^{A_{\sigma}^{\uparrow}(p+1)}(1)+\sum_{i=p+2}^n(x_{\sigma(i)}-x_{\sigma(i-1)})\,\delta_{f_0}^{A_{\sigma}^{\uparrow}(i)}(1)\nonumber\\
&& \null
+x_{\sigma(p)}\,\delta_{f_0}^{A_{\sigma}^{\downarrow}(p)}(1)+\sum_{i=1}^{p-1}(x_{\sigma(i)}-x_{\sigma(i+1)})\,\delta_{f_0}^{A_{\sigma}^{\downarrow}(i)}(1)\qquad
(\bfx\in\R^n_{\sigma}),\label{eq:G}
\end{eqnarray}
where $A_{\sigma}^{\downarrow}(i)=\{\sigma(1),\ldots,\sigma(i)\}$, with $A_{\sigma}^{\downarrow}(0)=\varnothing$, and the integer $p\in\{0,1,\ldots,n\}$ is such that $x_{\sigma(p)}<0\leqslant
x_{\sigma(p+1)}$. Thus we see that $f_0$ is homogeneous of degree one.

Nondecreasing symmetric Lov\'asz extensions vanishing at the origin, also called \emph{discrete symmetric Choquet integrals}, were introduced by
\v{S}ipo\v{s} \cite{Sip79} (see also \cite[{\S}5.4]{GraMarMesPap09}). We observe that a symmetric Lov\'asz extension $f\colon\R^n\to\R$ is a
symmetric Choquet integral if and only if its underlying pseudo-Boolean function $\phi=f|_{\{0,1\}^n}$ is nondecreasing and vanishes at the
origin.

%---------------------------------------------------------------------------------------------- Section 3
\section{Axiomatizations of Lov\'asz extensions}

In the present section we show that, for a class of intervals $I$, comonotonic additivity is equivalent to horizontal min-additivity (resp.\
horizontal max-additivity) and we describe the corresponding function class. By adding certain regularity conditions, we then axiomatize the
class of $n$-place Lov\'asz extensions. We first recall these properties.

Let $I$ be a nontrivial real interval, possibly unbounded, containing $0$. Two $n$-tuples $\bfx,\bfx'\in I^n$ are said to be \emph{comonotonic}
if there exists $\sigma\in S_n$ such that $\bfx,\bfx'\in I^n_{\sigma}$. A function $f\colon I^n\to\R$ is said to be \emph{comonotonically
additive} if, for every comonotonic $n$-tuples $\bfx,\bfx'\in I^n$ such that $\bfx+\bfx'\in I^n$, we have
\begin{equation}\label{eq:7sd6f8}
f(\bfx+\bfx')=f(\bfx)+f(\bfx').
\end{equation}
Given $\bfx\in I^n$ and $c\in I$, let $\llbracket\bfx\rrbracket_c=\bfx-\bfx\wedge c$ and $\llbracket\bfx\rrbracket^c=\bfx-\bfx\vee c$. We say
that a function $f\colon I^n\to\R$ is
\begin{itemize}
\item \emph{horizontally min-additive} if, for every $\bfx\in I^n$ and every $c\in I$ such that $\llbracket\bfx\rrbracket_c\in I^n$, we have
\begin{equation}\label{eq:s76f8}
f(\bfx)=f(\bfx\wedge c)+f(\llbracket\bfx\rrbracket_c).
\end{equation}

\item \emph{horizontally max-additive} if, for every $\bfx\in I^n$ and every $c\in I$ such that $\llbracket\bfx\rrbracket^c\in I^n$, we have
\begin{equation}\label{eq:xc987x}
f(\bfx)=f(\bfx\vee c)+f(\llbracket\bfx\rrbracket^c).
\end{equation}
\end{itemize}

We immediately observe that, since any $\bfx\in I^n$ decomposes into the sum of the comonotonic $n$-tuples $\bfx\wedge c$ and
$\llbracket\bfx\rrbracket_c$ for every $c$ (i.e., $\bfx=(\bfx\wedge c)+\llbracket\bfx\rrbracket_c$), any comonotonically additive function is
necessarily horizontally min-additive. Dually, any comonotonically additive function is horizontally max-additive.

We also observe that if $f\colon I^n\to\R$ satisfies any of these properties, then necessarily $f(\mathbf{0})=0$ (just take
$\bfx=\bfx'=\mathbf{0}$ and $c=0$ in (\ref{eq:7sd6f8})--(\ref{eq:xc987x})).

\begin{lemma}\label{lemma:sdaf786}
If $f\colon I^n\to\R$ is horizontally min-additive (resp.\ horizontally max-additive) then $\delta_f^A|_{I_+}$ (resp.\ $\delta_f^A|_{I_-}$) is
additive for every $A\subseteq [n]$. Moreover, if $I$ is centered at $0$, then $\delta_f$ is additive and odd.
\end{lemma}

\begin{proof}
We prove the result when $f$ is horizontally min-additive; the other claim can be dealt with dually. If $x,x'\in I_+$ is such that $x+x'\in
I_+$, then $x\leqslant x+x'$ and, using horizontal min-additivity with $c=x$, we get $\delta_f^A(x+x')=\delta_f^A(x)+\delta_f^A(x')$, which
shows that $\delta_f^A|_{I_+}$ is additive.

Assume now that $I$ is centered at $0$. If $x<0$ and $x'\geqslant 0$ are such that $x,x'\in I$, then $x\leqslant x+x'<x'$ and, using
horizontal min-additivity with $c=x$, we get $\delta_f(x+x')=\delta_f(x)+\delta_f(x')$. In particular, taking $x'=-x$, we obtain
$0=\delta_f(0)=\delta_f(x-x)=\delta_f(x)+\delta_f(-x)$ and hence $\delta_f(-x)=-\delta_f(x)$.

If $x<0$ and $x'<0$ are such that $x,x'\in I$, then, using horizontal min-additivity with $c=x+x'$, we get $\delta_f(x) =
\delta_f(x+x')+\delta_f(-x')=\delta_f(x+x')-\delta_f(x')$. Thus $\delta_f$ is additive and odd.
\end{proof}

\begin{remark}
For a horizontally min-additive or horizontally max-additive function $f\colon I^n\to\R$, the function $\delta_f^A$ need not be additive. For
instance, consider the horizontally min-additive function $f\colon\R^2\to\R$ defined by $f(x_1,x_2)=x_1\wedge x_2$. For $x>0$ and $x'=-x$, we
have
$$
\delta_f^{\{1\}}(x-x)=0>-x=\delta_f^{\{1\}}(x)+\delta_f^{\{1\}}(-x).
$$
\end{remark}

\begin{theorem}\label{thm:ads76fa}
Assume $0\in I\subseteq\R_+$ or $I=\R$. A function $f\colon I^n\to\R$ is horizontally min-additive if and only if there exists $g\colon
I^n\to\R$, with $\delta_g$ and $\delta_g^A|_{I_+}$ additive for every $A\subseteq [n]$, such that, for every $\sigma\in S_n$,
\begin{equation}\label{eq:sda78f9}
f(\bfx)=\delta_g(x_{\sigma(1)})+\sum_{i=2}^n \delta_g^{A_{\sigma}^{\uparrow}(i)}(x_{\sigma(i)}-x_{\sigma(i-1)})\qquad (\bfx\in I^n_{\sigma}).
\end{equation}
In this case, we can choose $g=f$.
\end{theorem}

\begin{proof}
(Necessity) Let $\sigma\in S_n$ and let $\bfx\in I^n_{\sigma}$. By repeatedly applying horizontal min-additivity with the successive cut levels
$$
x_{\sigma(1)},x_{\sigma(2)}-x_{\sigma(1)},\ldots,x_{\sigma(n-1)}-x_{\sigma(n-2)},
$$
we obtain
\begin{eqnarray*}
f(\bfx) &=& \delta_f(x_{\sigma(1)})+f(0,x_{\sigma(2)}-x_{\sigma(1)},\ldots,x_{\sigma(n)}-x_{\sigma(1)})\\
&=& \delta_f(x_{\sigma(1)})+\delta_f^{A_{\sigma}^{\uparrow}(2)}(x_{\sigma(2)}-x_{\sigma(1)})+ f(0,0,x_{\sigma(3)}-x_{\sigma(2)},\ldots,x_{\sigma(n)}-x_{\sigma(2)})\\
&=& \cdots\\
&=& \delta_f(x_{\sigma(1)})+\sum_{i=2}^n \delta_f^{A_{\sigma}^{\uparrow}(i)}(x_{\sigma(i)}-x_{\sigma(i-1)}).
\end{eqnarray*}
Thus (\ref{eq:sda78f9}) holds with $g=f$. Moreover, by Lemma~\ref{lemma:sdaf786}, $\delta_f$ and $\delta_f^A|_{I_+}$ are additive for every
$A\subseteq [n]$.

(Sufficiency) Let $\bfx\in I^n$ and $c\in I$ such that $\llbracket\bfx\rrbracket_c\in I^n$. There is $\sigma\in S_n$ such that $\bfx\in
I^n_{\sigma}$ and hence $f(\bfx)$ is given by (\ref{eq:sda78f9}), where $\delta_g$ and $\delta_g^A|_{I_+}$ are additive for every $A\subseteq
[n]$, which implies $g(\mathbf{0})=0$.

Suppose first that $c\leqslant x_{\sigma(1)}$. Then we have $f(\bfx\wedge c)=\delta_g(c)$ and
$$
f(\llbracket\bfx\rrbracket_c)= \delta_g(x_{\sigma(1)}-c)+\sum_{i=2}^n \delta_g^{A_{\sigma}^{\uparrow}(i)}(x_{\sigma(i)}-x_{\sigma(i-1)}).
$$
Since $\delta_g$ is additive, we finally obtain $f(\bfx\wedge c)+f(\llbracket\bfx\rrbracket_c)=f(\bfx)$.

Now, suppose that there is $p\in [n]$ such that $x_{\sigma(p)}<c\leqslant x_{\sigma(p+1)}$, where $x_{\sigma(n+1)}=\infty$. Then
$$
f(\bfx\wedge c)=\delta_g(x_{\sigma(1)})+\sum_{i=2}^p
\delta_g^{A_{\sigma}^{\uparrow}(i)}(x_{\sigma(i)}-x_{\sigma(i-1)})+\delta_g^{A_{\sigma}^{\uparrow}(p+1)}(c-x_{\sigma(p)})
$$
and
$$
f(\llbracket\bfx\rrbracket_c)=\delta_g^{A_{\sigma}^{\uparrow}(p+1)}(x_{\sigma(p+1)}-c)+\sum_{i=p+2}^n
\delta_g^{A_{\sigma}^{\uparrow}(i)}(x_{\sigma(i)}-x_{\sigma(i-1)}).
$$
Since $\delta_g^A|_{I_+}$ is additive for every $A\subseteq [n]$, we finally obtain $f(\bfx\wedge c)+f(\llbracket\bfx\rrbracket_c)=f(\bfx)$.
\end{proof}

Similarly, we obtain the following dual characterization.

\begin{theorem}\label{thm:ads76fab}
Assume $0\in I\subseteq\R_-$ or $I=\R$. A function $f\colon I^n\to\R$ is horizontally max-additive if and only if there exists $h\colon
I^n\to\R$, with $\delta_h$ and $\delta_h^A|_{I_-}$ additive for every $A\subseteq [n]$, such that, for every $\sigma\in S_n$,
$$
f(\bfx)=\delta_h(x_{\sigma(n)})+\sum_{i=1}^{n-1} \delta_h^{A_{\sigma}^{\downarrow}(i)}(x_{\sigma(i)}-x_{\sigma(i+1)})\qquad (\bfx\in
I^n_{\sigma}).
$$
In this case, we can choose $h=f$.
\end{theorem}

\begin{theorem}\label{thm:as876}
Assume $0\in I\subseteq\R_+$ or $I=\R$ (resp.\ $0\in I\subseteq\R_-$ or $I=\R$). A function $f\colon I^n\to\R$ is comonotonically additive if
and only if it is horizontally min-additive (resp.\ horizontally max-additive). In this case, $\delta_f$ and $\delta_f^A|_{I_+}$ (resp.\
$\delta_f$ and $\delta_f^A|_{I_-}$) are additive for every $A\subseteq [n]$.
\end{theorem}

\begin{proof}
We already observed that the condition is necessary. Let us now prove that it is sufficient.

Let $\bfx,\bfx'\in I^n$ be two comonotonic $n$-tuples, let $\sigma\in S_n$ be such that $\bfx,\bfx'\in I^n_{\sigma}$, and suppose that $f\colon
I^n\to\R$ is horizontally min-additive; the other case can be established dually. By Lemma~\ref{lemma:sdaf786} and Theorem~\ref{thm:ads76fa},
$\delta_f$ and $\delta_f^A|_{I_+}$ are additive for every $A\subseteq [n]$ and we have
\begin{eqnarray*}
\lefteqn{f(\bfx+\bfx')}\\
&=& \delta_f(x_{\sigma(1)}+x'_{\sigma(1)})+\sum_{i=2}^n \delta_f^{A_{\sigma}^{\uparrow}(i)}((x_{\sigma(i)}+x'_{\sigma(i)})-(x_{\sigma(i-1)}+x'_{\sigma(i-1)}))\\
&=& \delta_f(x_{\sigma(1)})+\delta_f(x'_{\sigma(1)})+ \sum_{i=2}^n \delta_f^{A_{\sigma}^{\uparrow}(i)}(x_{\sigma(i)}-x_{\sigma(i-1)}) + \sum_{i=2}^n \delta_f^{A_{\sigma}^{\uparrow}(i)}(x'_{\sigma(i)}-x'_{\sigma(i-1)})\\
&=& f(\bfx)+f(\bfx'),
\end{eqnarray*}
which shows that $f$ is comonotonically additive.
\end{proof}

\begin{remark}
\begin{enumerate}
\item[(a)] Theorems~\ref{thm:ads76fa} and \ref{thm:ads76fab} provide two equivalent representations of comonotonically additive functions
$f\colon\R^n\to\R$ (see Theorem~\ref{thm:as876}). For instance, for a binary comonotonically additive function $f\colon\R^2\to\R$, we have the
representations
$$
f(x_1,x_2)=
\begin{cases}
g(x_1,x_1)+g(0,x_2-x_1), & \mbox{if $x_1\leqslant x_2$},\\
g(x_2,x_2)+g(x_1-x_2,0), & \mbox{if $x_1\geqslant x_2$},
\end{cases}
$$
and
$$
f(x_1,x_2)=
\begin{cases}
h(x_2,x_2)+h(x_1-x_2,0), & \mbox{if $x_1\leqslant x_2$},\\
h(x_1,x_1)+h(0,x_2-x_1), & \mbox{if $x_1\geqslant x_2$},
\end{cases}
$$
where $\delta_g$, $\delta_g^A|_{\R_+}$, $\delta_h$, and $\delta_h^A|_{\R_-}$ are additive for every $A\subseteq [2]$. Figure~\ref{fig:ParaRule}
illustrates both representations in the region $x_1\leqslant x_2$, which recalls the standard ``parallelogram rule'' for vector addition. Thus,
$f$ is completely determined by its values on the $x_1$-axis, the $x_2$-axis, and the line $x_2=x_1$.

\setlength{\unitlength}{.07\linewidth}% .065\linewidth
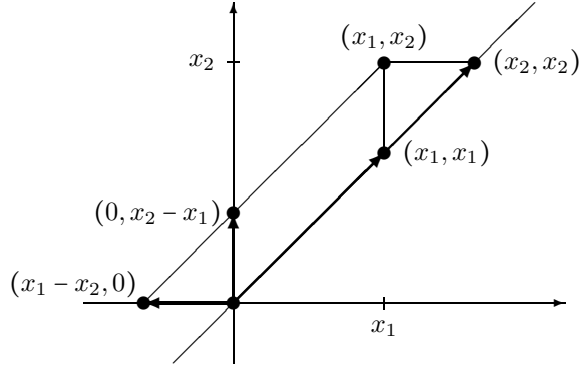
\begin{figure}[htb]
\begin{center}
\begin{picture}(8,6)
\put(2.5,0){\vector(0,1){6}}\put(0,1){\vector(1,0){8}}\put(1.5,0){\line(1,1){6}}\put(1,1){\line(1,1){4}}%
\put(5,3.5){\line(0,1){1.5}}\put(5,5){\line(1,0){1.5}}%
\put(1,1){\circle*{0.2}}\put(2.5,2.5){\circle*{0.2}}\put(5,3.5){\circle*{0.2}}\put(5,5){\circle*{0.2}}\put(6.5,5){\circle*{0.2}}%
\put(2.5,1){\circle*{0.2}}%
\put(5,0.9){\line(0,1){0.2}}\put(5,0.7){\makebox(0,0)[t]{$x_1$}}\put(2.4,5){\line(1,0){0.2}}\put(2.2,5){\makebox(0,0)[r]{$x_2$}}%
\put(5,5.2){\makebox(0,0)[b]{$(x_1,x_2)$}}\put(5.3,3.5){\makebox(0,0)[l]{$(x_1,x_1)$}}\put(6.8,5){\makebox(0,0)[l]{$(x_2,x_2)$}}%
\put(2.3,2.5){\makebox(0,0)[r]{$(0,x_2-x_1)$}}\put(0.9,1.1){\makebox(0,0)[br]{$(x_1-x_2,0)$}}%
\thicklines%
\put(2.5,1){\vector(-1,0){1.5}}\put(2.5,1){\vector(0,1){1.5}}\put(2.5,1){\vector(1,1){2.5}}\put(2.5,1){\vector(1,1){4}}
\end{picture}
\caption{Representations in the region $x_1\leqslant x_2$} \label{fig:ParaRule}
\end{center}
\end{figure}

\item[(b)] More generally, every comonotonically additive function $f\colon\R^n\to\R$ is completely determined by its values on the lines
$\{x\mathbf{1}_A : x\in\R\}$ ($A\subseteq [n]$).
\end{enumerate}
\end{remark}

We now axiomatize the class of $n$-place Lov\'asz extensions. A function $f\colon I^n\to\R$ is a Lov\'asz extension if it is the restriction to
$I^n$ of a Lov\'asz extension on $\R^n$.

\begin{theorem}\label{thm:saddsf687}
Assume $[0,1]\subseteq I\subseteq\R_+$ or $I=\R$. Let $f\colon I^n\to\R$ be a function and let $f_0=f-f(\mathbf{0})$. Then $f$ is a Lov\'asz
extension if and only if the following conditions hold:
\begin{enumerate}
\item[$(i)$] $f_0$ is comonotonically additive or horizontally min-additive (or horizontally max-additive if $I=\R$).

\item[$(ii)$] Each of the maps $\delta_{f_0}$ and $\delta_{f_0}^A|_{I_+}$ $(A\subseteq [n])$ is continuous at a point or monotonic or Lebesgue
measurable or bounded from one side on a set of positive measure.
\end{enumerate}
If $I=\R$, then the set $I_+$ can be replaced by $I_-$ in $(ii)$. Finally, Condition $(ii)$ holds whenever Condition $(i)$ holds and
$\delta_{f_0}^A$ is positively homogeneous of degree one for every $A\subseteq [n]$.
\end{theorem}

\begin{proof}
(Necessity) Follows from (\ref{eq:F}) and Theorems~\ref{thm:ads76fa} and \ref{thm:as876}.

(Sufficiency) By Theorems~\ref{thm:ads76fa} and \ref{thm:as876}, if $(i)$ holds, then $\delta_{f_0}$ and $\delta_{f_0}^A|_{I_+}$ are additive
for every $A\subseteq [n]$ and, for every $\sigma\in S_n$ and every $\bfx\in I^n_{\sigma}$, we have
$$%\begin{equation}\label{eq:dfsd74486}
f_0(\bfx)=\delta_{f_0}(x_{\sigma(1)})+\sum_{i=2}^n \delta_{f_0}^{A_{\sigma}^{\uparrow}(i)}(x_{\sigma(i)}-x_{\sigma(i-1)}).
$$%\end{equation}
By Theorem~\ref{thm:sa876as}, if $(ii)$ holds, then $\delta_{f_0}(x)=x\,\delta_{f_0}(1)$ for every $x\in I$ and
$\delta_{f_0}^A(x)=x\,\delta_{f_0}^A(1)$ for every $x\in I_+$. By (\ref{eq:F}) it follows that $f_0$ is a Lov\'asz extension such that
$f_0(\mathbf{0})=0$.

Now suppose that $I=\R$ and that $I_+$ is replaced by $I_-$ in $(ii)$. Then by Theorems~\ref{thm:ads76fab} and \ref{thm:as876}, $\delta_{f_0}$
and $\delta_{f_0}^A|_{I_-}$ are additive for every $A\subseteq [n]$ and, for every $\sigma\in S_n$ and every $\bfx\in I^n_{\sigma}$, we have
\begin{equation}\label{eq:dfsd786}
f_0(\bfx)=\delta_{f_0}(x_{\sigma(n)})+\sum_{i=1}^{n-1} \delta_{f_0}^{A_{\sigma}^{\downarrow}(i)}(x_{\sigma(i)}-x_{\sigma(i+1)}).
\end{equation}
By Theorem~\ref{thm:sa876as}, if $(ii)$ holds, then $\delta_{f_0}(x)=x\,\delta_{f_0}(1)$ for every $x\in I$ and
$\delta_{f_0}^A(x)=-x\,\delta_{f_0}^A(-1)$ for every $x\in I_-$. Thus (\ref{eq:dfsd786}) becomes
$$
f_0(\bfx)=x_{\sigma(n)}\,\delta_{f_0}(1)+\sum_{i=1}^{n-1} (x_{\sigma(i+1)}-x_{\sigma(i)})\,\delta_{f_0}^{A_{\sigma}^{\downarrow}(i)}(-1)
$$
from which we derive $\delta_{f_0}^{A_{\sigma}^{\uparrow}(i+1)}(1)=\delta_{f_0}(1)+\delta_{f_0}^{A_{\sigma}^{\downarrow}(i)}(-1)$ and hence
$$
f_0(\bfx) = x_{\sigma(n)}\,\delta_{f_0}(1)+\sum_{i=2}^{n}
(x_{\sigma(i)}-x_{\sigma(i-1)})\,\big(\delta_{f_0}^{A_{\sigma}^{\uparrow}(i)}(1)-\delta_{f_0}(1)\big).
$$
Again, we retrieve (\ref{eq:F}), thus showing that $f_0$ is a Lov\'asz extension such that $f_0(\mathbf{0})=0$.

Finally, if $(i)$ holds and $\delta_{f_0}^A$ is positively homogeneous of degree one for every $A\subseteq [n]$, then we have
$\delta_{f_0}(x)=x\,\delta_{f_0}(1)$ for every $x\in I_+$ and even for every $x\in I$ if $I=\R$ since then $\delta_{f_0}$ is odd by
Lemma~\ref{lemma:sdaf786}. Also, we have $\delta_{f_0}^A(x)=x\,\delta_{f_0}^A(1)$ for every $x\in I_+$ and, if $I=\R$,
$\delta_{f_0}^A(x)=-x\,\delta_{f_0}^A(-1)$ for every $x\in I_-$.
\end{proof}

\begin{remark}\label{rem:sdf76}
\begin{enumerate}
\item[(a)] Since any Lov\'asz extension vanishing at the origin is positively homogeneous of degree one, Condition $(ii)$ of
Theorem~\ref{thm:saddsf687} can be replaced by the stronger condition: $f_0$ is positively homogeneous of degree one.

\item[(b)] Axiomatizations of the class of $n$-place Choquet integrals can be immediately derived from Theorem~\ref{thm:saddsf687} by adding
nondecreasing monotonicity. Similar axiomatizations using comonotonic additivity (resp.\ horizontal min-additivity) were obtained by de Campos
and Bola\~{n}os \cite{deCBol92} (resp.\ by Benvenuti et al.~\cite[{\S}2.5]{BenMesViv02}).

\item[(c)] The concept of comonotonic additivity appeared first in Dellacherie~\cite{Del71} and then in Schmeidler~\cite{Sch86}. The concept of
horizontal min-additivity was previously considered by \v{S}ipo\v{s} \cite{Sip79} and then by Benvenuti et al.~\cite[{\S}2.3]{BenMesViv02} where
it was called ``horizontal additivity''.
\end{enumerate}
\end{remark}

%---------------------------------------------------------------------------------------------- Section 4
\section{Axiomatizations of symmetric Lov\'asz extensions}

In this final section we introduce a simultaneous generalization of horizontal min-additivity and horizontal max-additivity, called horizontal
median-additivity, and we describe the corresponding function class. By adding further conditions, we then axiomatize the class of $n$-place
symmetric Lov\'asz extensions.

Horizontal median-additivity in a sense combines horizontal min-additivity and horizontal max-additivity by using two cut levels that are
symmetric with respect to the origin. Formally, assuming that $I$ is centered at $0$, we say that a function $f\colon I^n\to\R$ is
\emph{horizontally median-additive} if, for every $\bfx\in I^n$ and every $c\in I_+$, we have
\begin{equation}\label{eq:sd98f7}
f(\bfx)=f\big(\mathrm{med}(-c,\bfx,c)\big)+f\big(\llbracket\bfx\rrbracket_c\big)+f\big(\llbracket\bfx\rrbracket^{-c}\big),
\end{equation}
where $\mathrm{med}(-c,\bfx,c)$ is the $n$-tuple whose $i$th component is the middle value of $\{-c,x_i,c\}$.

Since any $\bfx\in I^n$ decomposes into the sum of the comonotonic $n$-tuples $\mathrm{med}(-c,\bfx,c)+\llbracket\bfx\rrbracket^{-c}=\bfx\wedge
c$ and $\llbracket\bfx\rrbracket_c$ for every $c\in I_+$ (i.e.,
$\bfx=\mathrm{med}(-c,\bfx,c)+\llbracket\bfx\rrbracket_c+\llbracket\bfx\rrbracket^{-c}$), any comonotonically additive function is necessarily
horizontally median-additive. However, we will see (see Proposition~\ref{prop:xycv786} below) that the converse claim is not true.
%We will actually give a description of
%the class of horizontally median-additive functions and show that this class contains functions that are not comonotonically additive.

We also observe that if $f\colon I^n\to\R$ is horizontally median-additive, then necessarily $f(\mathbf{0})=0$ (take $\bfx=\mathbf{0}$ and $c=0$
in (\ref{eq:sd98f7})). We then see that
\begin{equation}\label{eq:s9fs87}
f(\bfx)=f(\bfx^+)+f(-\bfx^-)\qquad (\bfx\in I^n)
\end{equation}
(take $c=0$ in (\ref{eq:sd98f7})). This observation motivates the following definitions.

Assume that $I$ is centered at $0$. We say that a function $f\colon I^n\to\R$ is
\begin{itemize}
\item \emph{positively comonotonically additive} if (\ref{eq:7sd6f8}) holds for every comonotonic $n$-tuples $\bfx,\bfx'\in I_+^n$ such that
$\bfx+\bfx'\in I_+^n$.

\item \emph{negatively comonotonically additive} if (\ref{eq:7sd6f8}) holds for every comonotonic $n$-tuples $\bfx,\bfx'\in I_-^n$ such that
$\bfx+\bfx'\in I_-^n$.

\item \emph{positively horizontally min-additive} if (\ref{eq:s76f8}) holds for every $\bfx\in I_+^n$ and every $c\in I_+$.

\item \emph{negatively horizontally max-additive} if (\ref{eq:xc987x}) holds for every $\bfx\in I_-^n$ and every $c\in I_-$.
\end{itemize}

We observe immediately that if $f\colon I^n\to\R$ satisfies any of the four properties above, then $f(\mathbf{0})=0$.

\begin{lemma}\label{lemma:s9sd79s}
Assume that $I$ is centered at $0$. For any function $f\colon I^n\to\R$, the following assertions are equivalent.
\begin{enumerate}
\item[$(i)$] $f$ is horizontally median-additive.

\item[$(ii)$] $f$ is positively horizontally min-additive, negatively horizontally max-additive, and satisfies (\ref{eq:s9fs87}).

\item [$(iii)$] There exists a positively horizontally min-additive function $g\colon I^n\to\R$ and a negatively horizontally max-additive
function $h\colon I^n\to\R$ such that $f(\bfx)=g(\bfx^+)+h(-\bfx^-)$ for every $\bfx\in I^n$.
\end{enumerate}
\end{lemma}

\begin{proof}
$(i)\Rightarrow (ii)$ If $f$ is horizontally median-additive, then $f$ satisfies (\ref{eq:s9fs87}) and $f(\mathbf{0})=0$. Also, $f$ is
positively horizontally min-additive. Indeed, for every $\bfx\in I_+^n$ and every $c\in I_+$, we have $f(\bfx)=f(\bfx\wedge
c)+f(\llbracket\bfx\rrbracket_c)+f(\mathbf{0})$. Dually, we can also show that $f$ is negatively horizontally max-additive.

$(ii)\Rightarrow (iii)$ It suffices to take $g=h=f$.

$(iii)\Rightarrow (i)$ Assume that $(iii)$ holds. Then, for every $\bfx\in I^n$, we have $f(\bfx^+)=g(\bfx^+)$ and $f(-\bfx^-)=h(-\bfx^-)$ and
hence $f$ satisfies (\ref{eq:s9fs87}). Let $\bfx\in I^n$ and $c\in I_+$. Applying (\ref{eq:s9fs87}) to $\mathrm{med}(-c,\bfx,c)$, we obtain
\begin{equation}\label{eq:as5767as}
f\big(\mathrm{med}(-c,\bfx,c)\big)=f(\bfx^+\wedge c)+f\big((-\bfx^-)\vee(-c)\big).
\end{equation}
By (\ref{eq:s9fs87}) and positive horizontal min-additivity and negative horizontal max-additivity, we finally have
\begin{eqnarray*}
f(\bfx) &=& f(\bfx^+)+f(-\bfx^-)\\
&=& f(\bfx^+\wedge c)+f(\llbracket\bfx^+\rrbracket_c)+f\big((-\bfx^-)\vee(-c)\big)+f(\llbracket-\bfx^-\rrbracket^{-c})\\
&=& f\big(\mathrm{med}(-c,\bfx,c)\big)+f(\llbracket\bfx\rrbracket_c)+f(\llbracket\bfx\rrbracket^{-c})\qquad\mbox{(by (\ref{eq:as5767as}))},
\end{eqnarray*}
which shows that $f$ is horizontally median-additive.
\end{proof}

By Lemma~\ref{lemma:s9sd79s}, to describe the class of horizontally median-additive functions, it suffices to describe the class of positively
horizontally min-additive functions and that of negatively horizontally max-additive functions. These descriptions are given in the following
two theorems. The proofs are similar to those of Theorems~\ref{thm:ads76fa}, \ref{thm:ads76fab}, and \ref{thm:as876} and hence are omitted.

\begin{theorem}\label{thm:ads76fac}
Assume that $I$ is centered at $0$. For any function $f\colon I^n\to\R$, the following assertions are equivalent.
\begin{enumerate}
\item[$(i)$] $f$ is positively horizontally min-additive.

\item[$(ii)$] $f$ is positively comonotonically additive.

\item[$(iii)$] There exists $g\colon I^n\to\R$, with $\delta_g^A|_{I_+}$ additive for every $A\subseteq [n]$, such that, for every $\sigma\in
S_n$,
$$
f(\bfx)=\delta_g(x_{\sigma(1)})+\sum_{i=2}^n \delta_g^{A_{\sigma}^{\uparrow}(i)}(x_{\sigma(i)}-x_{\sigma(i-1)})\qquad (\bfx\in I^n_{\sigma}\cap
I^n_+).
$$
In this case, we can choose $g=f$.
\end{enumerate}
\end{theorem}

\begin{theorem}\label{thm:ads76facd}
Assume that $I$ is centered at $0$. For any function $f\colon\R^n\to\R$, the following assertions are equivalent.
\begin{enumerate}
\item[$(i)$] $f$ is negatively horizontally max-additive.

\item[$(ii)$] $f$ is negatively comonotonically additive.

\item[$(iii)$] There exists $h\colon I^n\to\R$, with $\delta_h^A|_{I_-}$ additive for every $A\subseteq [n]$, such that, for every $\sigma\in
S_n$,
$$
f(\bfx)=\delta_h(x_{\sigma(n)})+\sum_{i=1}^{n-1} \delta_h^{A_{\sigma}^{\downarrow}(i)}(x_{\sigma(i)}-x_{\sigma(i+1)})\qquad (\bfx\in
I^n_{\sigma}\cap I^n_-).
$$
In this case, we can choose $h=f$.
\end{enumerate}
\end{theorem}

The following theorem gives a description of the class of horizontally median-additive functions.

\begin{theorem}\label{thm:asds7sd798}
Assume that $I$ is centered at $0$. For any function $f\colon I^n\to\R$, the following assertions are equivalent.
\begin{enumerate}
\item[$(i)$] $f$ is horizontally median-additive.

\item[$(ii)$] $f$ is positively horizontally min-additive (or positively comonotonically additive), negatively horizontally max additive (or
negatively comonotonically additive), and satisfies (\ref{eq:s9fs87}).

\item[$(iii)$] There exist $g\colon I^n\to\R$ and $h\colon I^n\to\R$, with $\delta_g^A|_{I_+}$ and $\delta_h^A|_{I_-}$ additive for every
$A\subseteq [n]$, such that, for every $\sigma\in S_n$,
\begin{eqnarray}
f(\bfx) &=& \delta_g^{A_{\sigma}^{\uparrow}(p+1)}(x_{\sigma(p+1)})+\sum_{i=p+2}^n\delta_g^{A_{\sigma}^{\uparrow}(i)}(x_{\sigma(i)}-x_{\sigma(i-1)})\nonumber\\
&& \null +
\delta_h^{A_{\sigma}^{\downarrow}(p)}(x_{\sigma(p)})+\sum_{i=1}^{p-1}\delta_h^{A_{\sigma}^{\downarrow}(i)}(x_{\sigma(i)}-x_{\sigma(i+1)})\qquad
(\bfx\in I^n_{\sigma}),\label{eq:sadfsds}
\end{eqnarray}
where $p\in\{0,\ldots,n\}$ is such that $x_{\sigma(p)}<0\leqslant x_{\sigma(p+1)}$. In this case, we can choose $g=h=f$.
\end{enumerate}
\end{theorem}

\begin{proof}
$(i)\Leftrightarrow (ii)\Rightarrow (iii)$ Follows from Lemma~\ref{lemma:s9sd79s} and Theorems~\ref{thm:ads76fac} and \ref{thm:ads76facd}.

$(iii)\Rightarrow (ii)$ Considering (\ref{eq:sadfsds}) with $g=h=f$, we see immediately that $f$ satisfies (\ref{eq:s9fs87}). We then conclude
by Theorems~\ref{thm:ads76fac} and \ref{thm:ads76facd}.
\end{proof}

We now axiomatize the class of $n$-place symmetric Lov\'asz extensions. A function $f\colon I^n\to\R$ is a symmetric Lov\'asz extension if it is
the restriction to $I^n$ of a symmetric Lov\'asz extension on $\R^n$.

\begin{theorem}\label{thm:saddsf687a}
Assume that $I$ is centered at $0$ with $[-1,1]\subseteq I$. Let $f\colon I^n\to\R$ be a function and let $f_0=f-f(\mathbf{0})$. Then $f$
is a symmetric Lov\'asz extension if and only if the following conditions hold:
\begin{enumerate}
\item[$(i)$] $f_0$ is horizontally median-additive.

\item[$(ii)$] Each of the maps $\delta_{f_0}^A|_{I_+}$ and $\delta_{f_0}^A|_{I_-}$ $(A\subseteq [n])$ is continuous at a point or monotonic or
Lebesgue measurable or bounded from one side on a set of positive measure.

\item[$(iii)$] $\delta_{f_0}^A(-1)=-\delta_{f_0}^A(1)$ for every $A\subseteq [n]$.
\end{enumerate}
Conditions $(ii)$ and $(iii)$ hold together if and only if $\delta_{f_0}^A$ is homogeneous of degree one for every $A\subseteq [n]$.
\end{theorem}

\begin{proof}
(Necessity) Follows from (\ref{eq:G}) and Theorem~\ref{thm:asds7sd798}.

(Sufficiency) By Theorem~\ref{thm:asds7sd798}, if $(i)$ holds, then $\delta_{f_0}^A|_{I_+}$ and $\delta_{f_0}^A|_{I_-}$ are additive for every
$A\subseteq [n]$ and, for every $\sigma\in S_n$ and every $\bfx\in I^n_{\sigma}$, we have
\begin{eqnarray*}
f_0(\bfx) &=& \delta_{f_0}^{A_{\sigma}^{\uparrow}(p+1)}(x_{\sigma(p+1)})+\sum_{i=p+2}^n\delta_{f_0}^{A_{\sigma}^{\uparrow}(i)}(x_{\sigma(i)}-x_{\sigma(i-1)})\\
&& \null +
\delta_{f_0}^{A_{\sigma}^{\downarrow}(p)}(x_{\sigma(p)})+\sum_{i=1}^{p-1}\delta_{f_0}^{A_{\sigma}^{\downarrow}(i)}(x_{\sigma(i)}-x_{\sigma(i+1)}),
\end{eqnarray*}
where $p\in\{0,\ldots,n\}$ is such that $x_{\sigma(p)}<0\leqslant x_{\sigma(p+1)}$.

By Theorem~\ref{thm:sa876as}, if $(ii)$ holds, then $\delta_{f_0}^A(x)=x\,\delta_{f_0}^A(1)$ for every $x\in I_+$ and
$\delta_{f_0}^A(x)=-x\,\delta_{f_0}^A(-1)$ for every $x\in I_-$. Thus we have
\begin{eqnarray*}
f_0(\bfx) &=& x_{\sigma(p+1)}\,\delta_{f_0}^{A_{\sigma}^{\uparrow}(p+1)}(1)+\sum_{i=p+2}^n(x_{\sigma(i)}-x_{\sigma(i-1)})\,\delta_{f_0}^{A_{\sigma}^{\uparrow}(i)}(1)\\
&& \null
-x_{\sigma(p)}\,\delta_{f_0}^{A_{\sigma}^{\downarrow}(p)}(-1)-\sum_{i=1}^{p-1}(x_{\sigma(i)}-x_{\sigma(i+1)})\,\delta_{f_0}^{A_{\sigma}^{\downarrow}(i)}(-1).
\end{eqnarray*}
Using $(iii)$ and (\ref{eq:G}), it follows that $f_0$ is a symmetric Lov\'asz extension such that $f_0(\mathbf{0})=0$.

Finally, $(ii)$ and $(iii)$ imply that $\delta_{f_0}^A(x)=x\,\delta_{f_0}^A(1)$ and
$\delta_{f_0}^A(-x)=x\,\delta_{f_0}^A(-1)=-x\,\delta_{f_0}^A(1)$ for every $x\in I_+$, which means that $\delta_{f_0}^A$ is homogeneous of
degree one. Conversely, if $\delta_{f_0}^A$ is homogeneous of degree one for every $A\subseteq [n]$, then $(ii)$ and $(iii)$ hold trivially.
\end{proof}

\begin{remark}\label{rem:yx56y}
\begin{enumerate}
\item[(a)] Since any symmetric Lov\'asz extension vanishing at the origin is homogeneous of degree one, Conditions $(ii)$ and $(iii)$ of
Theorem~\ref{thm:saddsf687a} can be replaced by the stronger condition: $f_0$ is homogeneous of degree one.

\item[(b)] Axiomatizations of the class of $n$-place symmetric Choquet integrals can be immediately derived by adding nondecreasing
monotonicity.
\end{enumerate}
\end{remark}

The following proposition gives a condition for a symmetric Lov\'asz extension to be a Lov\'asz extension. This clearly shows that horizontal
median-additivity does not imply comonotonic additivity.

\begin{proposition}\label{prop:xycv786}
Let $f\colon\R^n\to\R$ be a Lov\'asz extension and let $g\colon\R^n\to\R$ be a symmetric Lov\'asz extension such that
$f|_{\{0,1\}^n}=g|_{\{0,1\}^n}$. Let also $f_0=f-f(\mathbf{0})$. Then we have $f=g$ if and only if $f_0(-\bfx)=-f_0(\bfx)$ for every
$\bfx\in\R_+^n$ (or equivalently, for every $\bfx\in\R_-^n$).
\end{proposition}

\begin{proof}
Since $f_0$ is comonotonically additive, for every $\bfx\in\R^n$ we have
$$
f(\bfx)=f(\bfx^+)+f(-\bfx^-)-f(\mathbf{0}).
$$
Combining this with (\ref{eq:sadfdfsa76}), that is
$$
g(\bfx)=f(\bfx^+)-f(\bfx^-)+f(\mathbf{0}),
$$
we see that $f=g$ if and only if $f_0(-\bfx^-)=-f_0(\bfx^-)$, which completes the proof.
\end{proof}

\section*{Acknowledgments}

The authors would like to thank Professor Maksa for useful discussions and his precious comments. This work was supported by the internal research project F1R-MTH-PUL-09MRDO of the University of Luxembourg.

%---------------------------------------------------------------------------------------------- Section 6
\appendix
\section{Proof of Theorem~\ref{thm:sa876as}}
\label{app}

We first observe that Theorem~\ref{thm:sa876as} holds when $I=\R$; see \cite[{\S}2]{AczDho89}. To see that the result still holds for any
nontrivial interval $I$ containing $0$, we consider the following two lemmas.

\begin{lemma}\label{lemma:app1}
Let $J$ be a nontrivial real interval containing $0$, and let $I=J+J=\{x+y : x,y\in J\}$. If a function $f\colon I\to\R$ satisfies
$f(x+x')=f(x)+f(x')$ for $x,x'\in J$, then $f$ can be uniquely extended onto $\R$ to an additive function.
\end{lemma}

\begin{proof}
See \cite[Theorem~13.5.3]{Kuc08}.
\end{proof}

\begin{lemma}\label{lemma:app2}
Let $J$ be a nontrivial real interval containing $0$, and let $I=J+J=\{x+y : x,y\in J\}$. If a function $f\colon I\to\R$ satisfies
$f(x+x')=f(x)+f(x')$ for $x,x'\in J$ and there is no $c\in\R$ such that $f(x)=cx$ for all $x\in J$, then the graph of $f$ is everywhere dense in
$I\times\R$.
\end{lemma}

\begin{proof}
By Lemma~\ref{lemma:app1}, there is an additive function $g\colon\R\to\R$ such that $g(x)=f(x)$ for $x\in I$. If there existed $c\in\R$ such
that $g(x)=cx$ for all $x\in\R$ then $f(x)=cx$ would follow for all $x\in I$. Since $0\in J$, we would have $J\subseteq I$ and hence $f(x)=cx$
for all $x\in J$, a contradiction. Therefore, since Theorem~\ref{thm:sa876as} holds when $I=\R$, the graph $G_g$ of $g$ must be dense in $\R^2$.
Hence, for any $(x,y)\in I\times\R$ there is a sequence $(x_n,g(x_n))\in G_g$ such that $(x_n,g(x_n))\to (x,y)$ as $n\to\infty$. since $I$ is an
interval we may (and do) assume that $x_n\in I$ for all $n$. Thus $(x_n,f(x_n))=(x_n,g(x_n))\to (x,y)$ as $n\to\infty$, which proves that the
graph of $f$ is dense in $I\times\R$.
\end{proof}

\begin{proof}[Proof of Theorem~\ref{thm:sa876as}]
Let $J=\{x/2 : x\in I\}$. By definition, $f$ satisfies $f(x+x')=f(x)+f(x')$ for $x,x'\in J$. By Lemma~\ref{lemma:app2}, if there is no $c\in\R$
such that $f(x)=cx$ for all $x\in J$, then the graph of $f$ is everywhere dense in $I\times\R$. Since Theorem~\ref{thm:sa876as} holds when
$I=\R$, the second claim follows trivially.
\end{proof}


\begin{thebibliography}{99}

\bibitem{AczDho89}
J.~Acz{\'e}l and J.~Dhombres.
\newblock {\em Functional equations in several variables}.
\newblock Encyclopedia of Mathematics and its Applications 31. Cambridge
  University Press, Cambridge, UK, 1989.

\bibitem{BelPraCal07}
G.~Beliakov, A.~Pradera, and T.~Calvo.
\newblock {\em Aggregation Functions: A Guide for Practitioners}.
\newblock Studies in Fuziness and Soft Computing. Springer, Berlin, 2007.

\bibitem{BenMesViv02}
P.~Benvenuti, R.~Mesiar, and D.~Vivona.
\newblock Monotone set functions-based integrals.
\newblock In {\em Handbook of measure theory, Vol. II}, pages 1329--1379.
  North-Holland, Amsterdam, 2002.

\bibitem{deCBol92}
L.~M. de~Campos and M.~J. Bola{\~n}os.
\newblock Characterization and comparison of {S}ugeno and {C}hoquet integrals.
\newblock {\em Fuzzy Sets and Systems}, 52(1):61--67, 1992.

\bibitem{Del71}
C.~Dellacherie.
\newblock Quelques commentaires sur les prolongements de capacit\'es.
\newblock In {\em S\'eminaire de Probabilit\'es, V (Univ. Strasbourg, ann\'ee
  universitaire 1969-1970)}, pages 77--81. Lecture Notes in Math., Vol. 191.
  Springer, Berlin, 1971.

\bibitem{GraMarMesPap09}
M.~Grabisch, J.-L. Marichal, R.~Mesiar, and E.~Pap.
\newblock {\em Aggregation functions}.
\newblock Encyclopedia of Mathematics and its Applications 127. Cambridge
  University Press, Cambridge, UK, 2009.

\bibitem{GraMurSug00}
M.~Grabisch, T.~Murofushi, and M.~Sugeno, editors.
\newblock {\em Fuzzy measures and integrals - Theory and applications},
  volume~40 of {\em Studies in Fuzziness and Soft Computing}.
\newblock Physica-Verlag, Heidelberg, 2000.

\bibitem{Kuc08}
M.~Kuczma.
\newblock {\em An introduction to the theory of functional equations and inequalities: Cauchy's equation and Jensen's inequality}.
\newblock Birkh\"auser, Basel, 2nd edition, 2008.

\bibitem{Lov83}
L.~Lov\'asz.
\newblock {Submodular functions and convexity}.
\newblock In {\em {Mathematical programming, 11th int. Symp., Bonn 1982,
  235--257}}. 1983.

\bibitem{Sch86}
D.~Schmeidler.
\newblock Integral representation without additivity.
\newblock {\em Proc. Amer. Math. Soc.}, 97(2):255--261, 1986.

\bibitem{Sin84}
I.~Singer.
\newblock {Extensions of functions of 0-1 variables and applications to
  combinatorial optimization}.
\newblock {\em Numer. Funct. Anal. Optimization}, 7:23--62, 1984.

\bibitem{Sip79}
J.~\v{S}ipo\v{s}.
\newblock Integral with respect to a pre-measure.
\newblock {\em Mathematica Slovaca}, 29(2):141--155, 1979.

\end{thebibliography}
\end{document}